\documentclass[]{interact}

\usepackage{epstopdf, color}
\usepackage[caption=false]{subfig}

\usepackage[numbers,sort&compress]{natbib}
\bibpunct[, ]{[}{]}{,}{n}{,}{,}
\makeatletter
\def\NAT@def@citea{\def\@citea{\NAT@separator}}
\makeatother

\theoremstyle{plain}
\newtheorem{theorem}{Theorem}[section]
\newtheorem{lemma}[theorem]{Lemma}
\newtheorem{corollary}[theorem]{Corollary}
\newtheorem{proposition}[theorem]{Proposition}

\theoremstyle{definition}

\DeclareMathOperator{\Circ}{Circ}

\theoremstyle{remark}
\newtheorem{remark}{Remark}

\begin{document}


\title{The $p$-norm of circulant matrices}

\author{
\name{Ludovick Bouthat\textsuperscript{a}\thanks{E-mail addresses: Ludovick.Bouthat.1@ulaval.ca (L. Bouthat), khare@iisc.ac.in (A. Khare), Javad.Mashreghi@mat.ulaval.ca (J. Mashreghi), Frederic.Morneau-Guerin@teluq.ca (F. Morneau-Gu\'erin).}, Apoorva Khare\textsuperscript{b,c}, Javad Mashreghi\textsuperscript{a}, and Fr\'ed\'eric Morneau-Gu\'erin\textsuperscript{$\ast$,a,d}}
\affil{\textsuperscript{$\ast$} Corresponding author.\linebreak \textsuperscript{a}D\'epartement de math\'ematiques et de statistique, Universit\'e Laval, 1045, avenue de la M\'edecine, Qu\'ebec (Qu\'ebec), G1V\;0A6, Canada.\linebreak \textsuperscript{b} Department of Mathematics, Indian Institute of Science, Bengaluru -- 560012, India.\linebreak
\textsuperscript{c} Analysis and Probability Research Group, Bengaluru -- 560012, India. \linebreak \textsuperscript{d}D\'epartement \'Education, Universit\'e T\'ELUQ, 455 rue du Parvis, Qu\'ebec (Qu\'ebec), G1K\;9H6, Canada.}
}

\maketitle

\begin{abstract}
In this note we study the induced $p$-norm of circulant matrices $A(n,\pm a, b)$, acting as operators on the Euclidean space $\mathbb{R}^n$. For circulant matrices whose entries are nonnegative real numbers, in particular for $A(n,a,b)$, we provide an explicit formula for the $p$-norm, $1 \leq p \leq \infty$. The calculation for $A(n,-a,b)$ is more complex. The 2-norm is precisely determined. As for the other values of $p$, two different categories of upper and lower bounds are obtained. These bounds are optimal at the end points (i.e. $p=1$ and $p = \infty$) as well as at $p=2$.
\end{abstract}


\begin{amscode}
15B05, 47A30
\end{amscode}

\section{Introduction and historical background}

Let $\alpha_1,\dots,\alpha_n$ be a sequence of real numbers. We denote by $A$ the circulant matrix defined as
\[
A = \Circ(\alpha_1, \alpha_2, \dots, \alpha_n) =
\begin{pmatrix}
\alpha_1 & \alpha_2 & \alpha_3 & \cdots & \alpha_n\\
\alpha_n & \alpha_1 & \alpha_2 & \cdots & \alpha_{n-1}\\
\alpha_{n-1} & \alpha_n & \alpha_1 & \cdots & \alpha_{n-2}\\
\vdots & \vdots & \vdots & \ddots & \vdots \\
\alpha_{2} & \alpha_3 & \alpha_4 & \cdots & \alpha_{1}
\end{pmatrix}.
\]

We endow $\mathbb{R}^n$ with the $p$-norm
\[
\| (x_1,x_2,\dots,x_n) \|_p = \left( \sum_{k=1}^{n} |x_k|^p \right)^{1/p}, \qquad 1 \leq p \leq \infty,
\]
and then we may interpret $A$ as an operator on $\mathbb{R}^n$ whose induced $p$-norm is given by
\[
\|A\|_p = \sup_{x \ne 0} \frac{\|Ax\|_p}{\|x\|_p}.
\]
In this note, our goal is to find $\|A\|_p$ or at least good estimations for it if the precise value cannot be determined.

In 1846, circulant matrices appeared somewhat implicitly in the work of Eugène Catalan \cite{1846-Catalan}. This appears to be -- according to Thomas Muir \cite[Vol 2., Ch. 14]{1906-Muir}, an authority on the early history of linear algebra and matrix theory --  their very first appearance in published mathematical work. An entire century passes, however, before the systematic study of circulant matrices begins to gain momentum. They have been put on a firm footing with the publication of a monograph by Davis \cite{1979-Davis} in 1979. It is not without reason that, in recent years, circulant matrices are still a topic of focus in matrix theory. Indeed, it is known that circulant and block-circulant matrices have a wide range of applications. For instance, they have been used in solving various ordinary and partial differential equations \cite{2013-Chen-Lin-Chen, 2005-Delgado-Romero-Rovelia-Vilamajo, 1983-Wilde}, in image processing \cite{2008-Wittsack-Wohlschlager-Ritzl-Kleiser-Cohnen-Seitz-Modder}, and in signal processing \cite{2008-Andrecut}.

In the last two decades, many scholars studied the norms of special circulant matrices whose entries are well-known number sequences. Due to the role that such matrices play in various disciplines such as statistics, numerical analysis, operator theory, mathematical physics and quantum information theory, it has become a very interesting research subject in matrix analysis and many authors have obtained some compelling results. Recently, researching the norms of some circulant-type matrices (e.g. $r$-circulant matrices, almost circulant matrices, geometric circulant matrices, skew-circulant matrices) has been one of the most interesting and active research areas in computational mathematics. It is a difficult task to provide a comprehensive list of contributions. We mention just a few below, which in fact mostly reflect our research interests.

Solak and Bozkurt \cite{2002-Solak-Bozkurt} obtained upper bounds for the entrywise $p$-matrix norm and the matrix norms induced by the vector $p$-norm of almost circulant, Cauchy--Toepliz and Cauchy--Hankel matrices. Bani-Domi and Kittaneh \cite{2008-BaniDomi-Kittaneh} have established two general norm equalities for circulant and skew circulant operator matrices. Kocer \cite{2007-Kocer} obtained norms of circulant, negacyclic and semi-circulant matrices with modified Pell, Jacobsthal and Jacobsthal--Lucas numbers. Solak \cite{2005-Solak} gave upper and lower bounds for the spectral and Frobenius norms of circulant matrices whose entries are classical Fibonacci and Lucas numbers. Ipek \cite{2011-Ipek} investigated some improved estimations for spectral norms of these matrices. The scope of the study initiated by Solak was broadened in various ways. For example, Nalli and Sen \cite{2010-Nalli-Sen} considered circulant matrices whose entries are generalized Fibonacci numbers whereas in \cite{2015-Bahsi, 2017-Chandoul, 2015-He-Ma-Zhang-Wang, 2017-Kome-Yazlik, 2010-Shen-Cen}, upper and lower bounds were obtained for the spectral norms of $r$-circulant matrices whose entries are either classical or generalized Fibonacci and Lucas numbers. Kizilates and Tuglu  \cite{2016-Kizilates-Tuglu} defined a geometric circulant matrix whose entries are the generalized Fibonacci numbers and hyperharmonic Fibonacci numbers, and gave upper and lower bounds for the spectral norms of this matrix. Kocer, Mansour and Tuglu \cite{2007-Kocer-Mansour-Tuglu} have presented the spectral norms and eigenvalues of circulant matrices whose entries are Horadam numbers. In the same vein, Yazlik and Taskara \cite{2013-Yazlik-Taskara} and Liu \cite{2014-Liu} obtained new upper and lower bounds for the spectral norms of an $r$-circulant matrix whose entries are the generalized $k$-Horadam numbers.

\section{Main results}

In this note, our main objective is to study the $p$-norm of a special case of circulant matrices, namely that of the form
\begin{equation}\label{E:defA-}
A(n,a,b)
=
\begin{pmatrix}
a & b & b & \cdots & b\\
b & a & b & \cdots & b\\
b & b & a & \cdots & b\\
\vdots & \vdots & \vdots & \ddots & \vdots \\
b & b & b & \cdots & a
\end{pmatrix},
\end{equation}
where $a,b \in \mathbb{R}$. Via a normalization process, it suffices to consider the following two cases: $A(n,a,b)$ and $A(n,-a,b)$, where $a,b \geq 0$.

Although our interest for matrices of the form $A(n,a,b)$ stems from other studies on doubly stochastic matrices, the results we obtained hold true in a more general setting. Our results show that the negative sign plays a crucial role as the $p$-norms of $A(n,a,b)$ and $A(n,-a,b)$ are entirely different. In fact, in Theorem \ref{T:p-norm-A+}, we show that
\[
\|A(n,a,b)\|_p = (n-1)b+a,
\]
for $1 \leq p \leq \infty$, while Theorem \ref{T:2-norm-A-} states that
\[
\|A(n,-a,b)\|_2 =
\left\{
\begin{array}{lcl}
a+b & \mbox{if} & (n-2)b \leq 2a, \\
   &  & \\
(n-1)b-a   & \mbox{if} & (n-2)b \geq 2a.
\end{array}
\right.
\]

We did not succeed to precisely evaluate $\|A(n,-a,b)\|_p$, for $p \ne 2$. However, using two different methods, we obtained some upper and lower bounds which are optimal at the end points. In Theorem \ref{T:2-norm-A-p}, we show
\[
\left\{
\begin{array}{ccl}
a+b \leq \|A(n,-a,b)\|_p \leq n^{1/2-1/p}(a+b) & \mbox{if} & (n-2)b \leq 2a, \\
   &  & \\
(n-1)b-a \leq \|A(n,-a,b)\|_p \leq n^{1/2-1/p} ((n-1)b-a)   & \mbox{if} & (n-2)b \geq 2a.
\end{array}
\right.
\]
In Theorem \ref{T:3-norm-A-p}, we refine these estimate, albeit with more complex formulas, as
\[
\left\{
\begin{array}{ccl}
a+b \leq \|A(n,-a,b)\|_p \leq \big(a+b\big)^{\frac{2}{p}} \big((n-1)b+a \big)^{1-\frac{2}{p}} & \mbox{if} & (n-2)b \leq 2a, \\
   &  & \\
(n-1)b-a \leq \|A(n,-a,b)\|_p \leq \big((n-1)b-a \big)^{\frac{2}{p}} \big((n-1)b+a\big)^{1- \frac{2}{p}}   & \mbox{if} & (n-2)b \geq 2a.
\end{array}
\right.
\]
All formulas are for $p \geq 2$. But since $A=A(n,\pm a,b)$ is self-adjoint, we have ${\|A\|_p = \|A\|_q}$ where $1/p+1/q=1$ (see for instance \cite[Theorem 5.6.36]{1985-Horn-Johnson}). Hence it is enough to study one of the cases among $p \in [1,2]$ and $p \in [2, \infty]$.

\section{The $p$-norm of the circulant matrix $A(n,a,b)$}

In this section, we find the $p$-norm of the circulant matrix $A(n,a,b)$ for $a, b \geq 0$, and, more generally, of all circulant matrices with nonnegative entries. In the next section, we will consider circulant matrices with negative entries and will immediately notice that the calculations are more involved.

To determine the $p$-norm of $A(n,a,b)$, first note that the vector $x=(1,1,\dots,1)$ tells us
\begin{equation}\label{lower-bound-Anab}
\|A(n,a,b)\|_p \geq \frac{\|A(n,a,b)x\|_p}{\|x\|_p} = (n-1)b + a.
\end{equation}
Now, we show that $x=(1,1,\dots,1)$ is actually a maximizing vector, which means that we can replace $\geq$ by the identity in the above equation. To do so, we need the following inequality which is interesting in its own right.

\begin{lemma} \label{L:p-inegalite}
Let $\alpha_1,\dots,\alpha_n$ be a sequence of nonnegative numbers.  Let $x_1,\dots,x_n$ be another sequence of nonnegative numbers. Let $S_n$ denote the symmetric group of permutations of $\{1,2,\dots,n\}$. Assume $p > 1$. Then
\[
\sum_{\sigma \in S_n} \big( \alpha_{\sigma(1)}x_1 + \cdots + \alpha_{\sigma(n)}x_n \big)^p
\leq
 (n-1)! (\alpha_1 + \dots + \alpha_n)^{p}\left(x_1^p + \cdots + x_n^p \right).
\]
\end{lemma}

\begin{proof}
Write
\[
\alpha_{\sigma(1)}x_1 + \cdots + \alpha_{\sigma(n)}x_n = \alpha_{\sigma(1)}^{1/q} \cdot \alpha_{\sigma(1)}^{1/p} x_1 + \cdots + \alpha_{\sigma(n)}^{1/q} \cdot \alpha_{\sigma(n)}^{1/p}x_n.
\]
Then, by H\"{o}lder's inequality,
\begin{align*}
	\alpha_{\sigma(1)}x_1 + \cdots + \alpha_{\sigma(n)}x_n &\leq \left( \alpha_{\sigma(1)} + \cdots + \alpha_{\sigma(n)}\right)^{1/q} \left(\alpha_{\sigma(1)} x_1^p + \cdots + \alpha_{\sigma(n)}x_n^p \right)^{1/p}\\
	&=  (\alpha_1 + \dots + \alpha_n)^{1/q} \left(\alpha_{\sigma(1)} x_1^p + \cdots + \alpha_{\sigma(n)}x_n^p \right)^{1/p}.
\end{align*}
Therefore,
\begin{align*}
\sum_{\sigma \in S_n} \big( \alpha_{\sigma(1)}x_1 + \cdots + \alpha_{\sigma(n)}x_n \big)^p
&\leq  (\alpha_1 + \dots + \alpha_n)^{p/q}\sum_{\sigma \in S_n}\left(\alpha_{\sigma(1)} x_1^p + \cdots + \alpha_{\sigma(n)}x_n^p \right)\\
&=  (n-1)! (\alpha_1 + \dots + \alpha_n)^{p}\left(x_1^p + \cdots + x_n^p \right).
\end{align*}
\end{proof}

\begin{remark}
Of particular interest is the case where $\alpha_1, \dots, \alpha_n$ is a convex sequence, i.e., nonnegative and add up to 1. We obtain the following identity:
\[
\frac{1}{n!} \sum_{\sigma \in S_n} \big( \alpha_{\sigma(1)}x_1 + \cdots + \alpha_{\sigma(n)}x_n \big)^p
\leq
\frac{x_1^p+\cdots+x_n^p}{n}.
\]
\end{remark}

We can now find the precise value of $\|A(n,a,b)\|_p$.

\begin{theorem} \label{T:p-norm-A+}
Let $a,b \geq 0$, and let $A= A(n,a,b)$ be given by \eqref{E:defA-}. Assume $p > 1$.  Then
\[
\|A\|_p = (n-1)b+a.
\]
\end{theorem}

\begin{proof}
For the special pattern $(\alpha_1,\alpha_2,\dots,\alpha_n) = (a,b,\dots,b)$,  we have 
\begin{eqnarray*}
&&\frac{1}{(n-1)!}\sum_{\sigma \in S_n} \big( \alpha_{\sigma(1)}x_1 + \alpha_{\sigma(2)}x_2 + \cdots + \alpha_{\sigma(n)}x_n \big)^p\\
&=& \big( ax_1 + bx_2+ \cdots + bx_n \big)^p\\
&+& \big( bx_1 + ax_2+ \cdots + bx_n \big)^p\\
&\vdots&\\
&+& \big( bx_1 + bx_2+ \cdots + ax_n \big)^p.
\end{eqnarray*}
Since
\[
\begin{pmatrix}
a & b& b & \cdots & b\\
b & a & b & \cdots & b\\
b & b & a & \cdots & b\\
\vdots & \vdots & \vdots & \ddots & \vdots \\
b & b & b & \cdots & a
\end{pmatrix}
\begin{pmatrix}
x_1\\
x_2\\
x_3\\
\vdots\\
x_n
\end{pmatrix}
=
\begin{pmatrix}
ax_1 + bx_2+ bx_3+ \cdots + bx_n\\
bx_1 + ax_2+ bx_3+ \cdots + bx_n\\
bx_1 + bx_2+ ax_3+ \cdots + bx_n\\
\vdots\\
bx_1 + bx_2+ bx_3+ \cdots + ax_n
\end{pmatrix},
\]
by Lemma \ref{L:p-inegalite}, we conclude that
\[
\frac{\|Ax\|_p}{\|x\|_p} \leq (n-1)b+a.
\]
Therefore, in the light of \eqref{lower-bound-Anab},
\[
\|A\|_p =  (n-1)b+a.
\]
\end{proof}

It is not difficult to generalize the above method and find the $p$-norm of a general circulant matrix $A = \Circ(\alpha_1, \alpha_2, \dots, \alpha_n)$ with nonnegative entries. Let $y=Ax$ and for the sake of brevity write $\Delta = \alpha_1+\alpha_2+\cdots+\alpha_n$. Then, as in the proof of Lemma \ref{L:p-inegalite}, we have
\begin{eqnarray*}
y_1 &\leq&  \Delta^{1/q} (\alpha_1 x_1^p + \alpha_2 x_2^p+ \alpha_3 x_3^p+ \cdots + \alpha_n x_n^p)^{1/p},\\
y_2 &\leq&  \Delta^{1/q} (\alpha_n x_1^p + \alpha_1 x_2^p+ \alpha_2 x_3^p+ \cdots + \alpha_{n-1}x_n^p)^{1/p},\\
&\vdots&\\
y_n &\leq& \Delta^{1/q} (\alpha_2 x_1^p + \alpha_3 x_2^p+ \alpha_4 x_3^p+ \cdots + \alpha_1 x_n^p)^{1/p}.
\end{eqnarray*}
Therefore,
\[
\sum_{k=1}^{n}y_k^p \leq \Delta^{p} \sum_{k=1}^{n}x_k^p.
\]
This estimation means $\|Ax\|_p \leq \Delta \, \|x\|_p$. Considering again the same maximizing vector $x=(1,1,\dots,1)$, we conclude that

\begin{equation}
\|A\|_p = \alpha_1+\alpha_2+\cdots+\alpha_n.
\end{equation}

\section{The $2$-norm of the circulant matrix $A(n,-a,b)$}

Recall the definition of $A(n,-a,b)$ from \eqref{E:defA-}. As is usually the case, calculations for $p=2$ are easier. Nevertheless, even in this case there is an unexpected outcome: the 2-norm is not given by a single formula!

\begin{theorem} \label{T:2-norm-A-}
Let $a,b \geq 0$, and let $A=A(n,-a,b)$ be given by \eqref{E:defA-}. Then
\[
\|A\|_2 =
\left\{
\begin{array}{lcl}
a+b & \mbox{if} & (n-2)b \leq 2a, \\
   &  & \\
(n-1)b-a   & \mbox{if} & (n-2)b \geq 2a.
\end{array}
\right.
\]
\end{theorem}

\begin{proof}
For the sake of notational simplicity, in this proof we write $A$ for $A(n,-a,b)$. Let $y=Ax$, i.e.,
\[
\begin{pmatrix}
y_1\\
y_2\\
y_3\\
\vdots\\
y_n
\end{pmatrix}
=
\begin{pmatrix}
-a x_1 + b x_2 + b x_3+ \cdots + b x_n\\
b x_1 - a x_2 + b x_3+ \cdots + bx_n\\
b x_1 + b x_2 -a x_3+ \cdots + b x_n\\
\vdots\\
b x_1 + b x_2+ b x_3+ \cdots -a x_n
\end{pmatrix}.
\]
Then
\begin{eqnarray*}
\|Ax\|_2^2 &=&  y_1^2+y_2^2+\cdots+y_n^2\\
&=& (-a x_1 + b x_2 + b x_3+ \cdots + b x_n)^2\\
&+& (b x_1 - a x_2 + b x_3+ \cdots + bx_n)^2 \\
&+& (b x_1 + b x_2 -a x_3+ \cdots + b x_n)^2 \\
&\vdots&\\
&+& (b x_1 + b x_2+ b x_3+ \cdots -a x_n)^2 \\
&=& (a^2+(n-1)b^2) \|x\|_2^2 + (2(n-2)b^2-4ab) (x_1x_2+x_1x_3+\cdots+x_{n-1}x_{n}).
\end{eqnarray*}
Since
\[
(x_1+x_2+\cdots+x_n)^2= \|x\|_2^2 +2(x_1x_2+x_1x_3+\cdots+x_{n-1}x_{n}),
\]
we can rewrite the above identity as
\begin{equation}\label{E:normp2m}
\|Ax\|_2^2 =  (a+b)^2 \|x\|_2^2 + b((n-2)b-2a) (x_1+x_2+\cdots+x_{n})^2.
\end{equation}
With the normalizing assumption $x_1^2+x_2^2+\cdots+x_n^2=1$, we have
\[
\|Ax\|_2^2 =  (a+b)^2 + b((n-2)b-2a) (x_1+x_2+\cdots+x_{n})^2.
\]
Now, we face two cases.\\

\noindent {\em Case I:} $(n-2)b \leq 2a$. In this case
\[
\|Ax\|_2^2 \leq  (a+b)^2
\]
and the maximizing vectors on the unit sphere are precisely those with
\[
x_1+x_2+\cdots+x_{n}=0.
\]
Hence,
\[
\|A\|_2 = a+b.\\
\]

\noindent {\em Case II:} $(n-2)b \geq 2a$. In this case, the combination $(x_1+x_2+\cdots+x_{n})^2$ plays a role and its maximum on the unit sphere is needed. By direct verification (or by Cauchy--Schwarz), the maximum happens whenever
\[
x_1=x_2=\cdots=x_n= \frac{1}{\sqrt{n}}.
\]
Thus,
\[
\|Ax\|_2^2 \leq  (a+b)^2 + n b((n-2)b-2a) = (a-(n-1)b)^2,
\]
and the maximizing vectors on the unit sphere are $\pm \frac{1}{\sqrt{n}} (1,1,\dots,1)$. Hence,
\[
\|A\|_2 = (n-1)b-a.
\]
\end{proof}

We shall now present a generalization of the case $A(3,-a,b)$.

\begin{proposition} \label{T:2-norm-Circ_3}
If $A=\Circ(\alpha_1, \alpha_2, \alpha_3)$ for arbitrary $\alpha_1, \alpha_2, \alpha_3 \in \mathbb{R}$, then
\[
\|A\|_2 =
\left\{
\begin{array}{lcl}
\sqrt{\alpha_1^2 + \alpha_2^2 + \alpha_3^2 - (\alpha_1\alpha_2+\alpha_2\alpha_3+\alpha_3\alpha_1)} & \mbox{if} & \alpha_1\alpha_2+\alpha_2\alpha_3+\alpha_3\alpha_1 \leq 0, \\
   &  & \\
|\alpha_1+\alpha_2+\alpha_3|   & \mbox{if} & \alpha_1\alpha_2+\alpha_2\alpha_3+\alpha_3\alpha_1 \geq 0.
\end{array}
\right.
\]
\end{proposition}

Note that the special case $b= \alpha_2 = \alpha_3 \geq 0$ and $-a= \alpha_1 < 0$ yields the case $A(3,-a,b)$.

The proof of Proposition \ref{T:2-norm-Circ_3} is omitted since we shall now state and prove a more general theorem that subsumes both previous results and that conceptually explains why there are two cases that arise in them.

\begin{theorem} \label{T:2-norm-Circ_n}
Suppose $A = [c_1 | c_2 | \dots | c_n] \in \mathbb{R}^{n \times n}$ is an arbitrary matrix whose columns $c_1, \dots, c_n \in \mathbb{R}^n$ satisfy:
\[
c_1^\intercal  c_1 = \dots = c_n^\intercal  c_n = \rho, \quad\quad c_i^\intercal c_j = \beta,\quad (\forall 1 \leq i < j \leq n),
\]
for some scalars $\rho, \beta \in \mathbb{R}$.
\[
\|A\|_2 =
\left\{
\begin{array}{lcl}
\sqrt{\rho - \beta} & \mbox{if} & \beta \leq 0, \\
   &  & \\
\sqrt{\rho + (n-1)\beta}  & \mbox{if} & \beta \geq 0.
\end{array}
\right.
\]
\end{theorem}

\begin{proof}
It is well known that the square of the induced 2-norm (also called the spectral norm) of any real $n \times n$ matrix $A$ is precisely the largest eigenvalue of the matrix $A^\intercal A$. But, under the above-stated hypotheses,
\[
A^\intercal A =
\begin{pmatrix}
\rho & \beta & \beta & \cdots & \beta\\
\beta & \rho & \beta & \cdots & \beta\\
\beta & \beta & \rho & \cdots & \beta\\
\vdots & \vdots & \vdots & \ddots & \vdots \\
\beta & \beta & \beta & \cdots & \rho
\end{pmatrix}
= (\rho - \beta)I + \beta K,
\]
where $I$ is the $n \times n$ identity matrix and $K$ is the $n\times n$ all-ones matrix (i.e., the matrix where every element is equal to 1). The eigenvalues of the positive semidefinite matrix $A^\intercal A =  (\rho - \beta)I + \beta K$ are $\rho - \beta$ (with multiplicity of $n-1$) and $\rho + (n-1)\beta$ (with multiplicity 1).Taking their maximum and then its square root, the result follows.
\end{proof}

Remark that Theorem \ref{T:2-norm-A-} corresponds to the case where $\rho = a^2 + (n-1)b^2$ and $\beta = -2ab + (n-2)b^2$, whereas Proposition \ref{T:2-norm-Circ_3} corresponds to the case where $n = 3$, $\rho = \alpha_1^2 + \alpha_2^2 + \alpha_3^2$ and $\beta = \alpha_1\alpha_2 + \alpha_2\alpha_3 + \alpha_3\alpha_1$.

\section{Estimation of the $p$-norm of the circulant matrix $A(n,-a,b)$}

When $p \ne 2$, we have not been able to obtain a precise formula for the $p$-norm of $A(n,-a,b)$. However, using two different techniques, we provide certain lower and upper bounds.

\begin{theorem} \label{T:2-norm-A-p}
Let $a,b \geq 0$, and let $A=A(n,-a,b)$ be given by \eqref{E:defA-}.
Assume $p$ verifies ${2 \leq p < \infty}$. Then
\[
\left\{
\begin{array}{ccl}
a+b \leq \|A\|_p \leq n^{\frac{1}{2}-\frac{1}{p}}(a+b) & \mbox{if} & (n-2)b \leq 2a, \\
   &  & \\
(n-1)b-a \leq \|A\|_p \leq n^{\frac{1}{2}-\frac{1}{p}} ((n-1)b-a)   & \mbox{if} & (n-2)b \geq 2a.
\end{array}
\right.
\]
\end{theorem}

\begin{proof}
Assume that $|x_1|^p+|x_2|^p+\cdots+|x_n|^p=1$. Then
\begin{eqnarray*}
\|Ax\|_p^p &=&  |y_1|^p+|y_2|^p+\cdots+|y_n|^p\\
&=& |-a x_1 + b x_2 + b x_3+ \cdots + b x_n|^p\\
&+& |b x_1 - a x_2 + b x_3+ \cdots + bx_n|^p \\
&+& |b x_1 + b x_2 -a x_3+ \cdots + b x_n|^p \\
&\vdots&\\
&+& |b x_1 + b x_2+ b x_3+ \cdots -a x_n|^p \\
&=& |nbS-(a+b)x_1|^p+|nbS-(a+b)x_2|^p+\cdots+|nbS-(a+b)x_n|^p,
\end{eqnarray*}
where $S=(x_1+x_2+\cdots+x_n)/n$. The points for which $|x_1|^p+|x_2|^p+\cdots+|x_n|^p=1$ and $S=0$ imply
\begin{equation}\label{E:estimate-1}
\|A\|_p \geq a+b.
\end{equation}
Moreover, the choices $\pm \frac{1}{\sqrt[p]{n}}\, (1,1,1,\dots,1)$ give
\begin{equation}\label{E:estimate-2}
\|A\|_p \geq |(n-1)b-a|.
\end{equation}
Note that the above results are valid for any $p>1$.

Now, by \eqref{E:normp2m} and under the assumption that $p \geq 2$, we have
\[
\|Ax\|_p \leq \|Ax\|_2 = \left( (a+b)^2 \|x\|_2^2 + b((n-2)b-2a) (x_1+x_2+\cdots+x_{n})^2 \right)^{1/2}.
\]
Recall that H\"older's inequality implies that
\[
\|x\|_p ~\leq~ \|x\|_2 \leq n^{\frac{1}{2}-\frac{1}{p}}\|x\|_p,
\]
which we may rewrite as $\|x\|_2^2 \leq n^{1-\frac{2}{p}}\|x\|_p^{2}$. Hence
\[
\|Ax\|_p \leq \left( (a+b)^2 n^{1-\frac{2}{p}}\|x\|_p^{2} + b((n-2)b-2a) (x_1+x_2+\cdots+x_{n})^2 \right)^{1/2}.
\]

\noindent {\em Case I:} $(n-2)b \leq 2a$. In this case
\[
\|Ax\|_p^2 \leq  (a+b)^2 n^{1-\frac{2}{p}}\|x\|_p^{2},
\]
which implies
\[
\|A\|_p \leq n^{\frac{1}{2}-\frac{1}{p}}(a+b).
\]
Therefore, by \eqref{E:estimate-1}, we have
\[
(a+b) \leq \|A\|_p \leq n^{\frac{1}{2}-\frac{1}{p}}(a+b).
\]

\noindent {\em Case II:} $(n-2)b \geq 2a$. In this case, the maximum of $(x_1+x_2+\cdots+x_{n})^2$ on the unit sphere of $(\mathbb{R}^n, \|\cdot\|_p)$ is needed. By direct verification, this happens whenever
\[
x_1=x_2=\cdots=x_n= \frac{1}{\sqrt[p]{n}}.
\]
Thus,
\[
\|Ax\|_p^2 \leq  (a+b)^2 n^{1-\frac{2}{p}} + b((n-2)b-2a) n^{1-\frac{2}{p}},
\]
which implies
\[
\|A\|_p \leq n^{\frac{1}{2}-\frac{1}{p}} ((n-1)b-a).
\]
Therefore, by \eqref{E:estimate-2}, we have
\[
(n-1)b-a \leq \|A\|_p \leq n^{\frac{1}{2}-\frac{1}{p}} ((n-1)b-a).
\]
\end{proof}

In the following, we use the Riesz--Thorin interpolation theorem \cite[Theorem 1.1.1.]{1976-Lofstrom} to obtain another set of upper bounds for $\|A\|_p$. Even though the combinations below are more complex than those given in Theorem \ref{T:2-norm-A-p}, they provide more accurate estimations, which are optimal for $p = \infty$.

\begin{theorem} \label{T:3-norm-A-p}
Let $a,b \geq 0$, and let $A=A(n,-a,b)$ be given by \eqref{E:defA-}.
Assume $p$ verifies $2 \leq p \leq \infty$. Then

\[
\left\{
\begin{array}{ccl}
a+b \leq \|A\|_p \leq \big(a+b\big)^{\frac{2}{p}} \big((n-1)b+a \big)^{1-\frac{2}{p}} & \mbox{if} & (n-2)b \leq 2a, \\
   &  & \\
(n-1)b-a \leq \|A\|_p \leq \big((n-1)b-a \big)^{\frac{2}{p}} \big((n-1)b+a\big)^{1- \frac{2}{p}}   & \mbox{if} & (n-2)b \geq 2a.
\end{array}
\right.
\]
Both upper bounds are optimal for $p = \infty$.
\end{theorem}

\begin{proof}
First observe that
\[
\|A\|_\infty ~=~ \max\limits_{1 \leq i \leq n} \sum\limits_{j=1}^n |A_{i,j}| = (n-1)b+a.
\]
In particular, this means that both upper bounds are optimal for $p = \infty$.

That $\|A\|_p \geq a+b$ and $\|A\|_p \geq |(n-1)b-a|$ was proved in \eqref{E:estimate-1} and \eqref{E:estimate-2} for ${2 \leq p < \infty}$.  Now, by applying the Riesz--Thorin interpolation theorem with ${p_0 = q_0 = 2}$,  ${p_1 = q_1 = \infty}$ and $q_\theta = p_\theta$, we obtain 
\begin{equation}\label{E:Interpol}
\|A\|_{p_\theta} ~\leq~ \|A\|_2^{1-\theta} \|A\|_\infty^\theta
\end{equation}
for (with a mild abuse of notation)
\[
\frac{1}{p_\theta} = \frac{1-\theta}{2} + \frac{\theta}{\infty} ~=~  \frac{1-\theta}{2}, \quad\quad (0< \theta< 1).
\]
Remark that $\theta = 1 - \frac{2}{p_\theta}$. Hence (\ref{E:Interpol}) can be restated as
\begin{equation*}\begin{aligned}
 \|A\|_{p_\theta} ~&\leq~ \|A\|_2^{\frac{2}{p_\theta}} \|A\|_\infty^{1- \frac{2}{p_\theta}}, \quad\quad (2 < p_\theta < \infty).
\end{aligned}\end{equation*}
It follows from Theorem \ref{T:2-norm-A-} that one needs to consider two cases:

\noindent {\em Case I:} $(n-2)b \leq 2a$. In this case,
\[
\|A\|_{p_\theta} ~\leq~ \big(a+b\big)^{\frac{2}{p}} \big((n-1)b+a \big)^{1-\frac{2}{p}}.
\]

\noindent {\em Case II:} $(n-2)b \geq 2a$. In this case,
\[
\|A\|_{p_\theta}  ~\leq~ \big((n-1)b-a \big)^{\frac{2}{p}} \big((n-1)b+a\big)^{1- \frac{2}{p}}.
\]
\end{proof}

\begin{corollary} \label{C:3-norm-A-p}
Let $a,b \geq 0$, and let $A=A(n,-a,b)$ be given by \eqref{E:defA-}.
Assume $p$ verifies ${1 \leq p \leq 2}$. Then, 
\[
\left\{
\begin{array}{ccl}
a+b \leq \|A\|_p \leq \big((n-1)b+a \big)^{\frac{2}{p}-1} \big(a+b\big)^{2- \frac{2}{p}} & \mbox{if} & (n-2)b \leq 2a, \\
   &  & \\
(n-1)b-a \leq \|A\|_p \leq \big((n-1)b+a \big)^{\frac{2}{p}-1} \big((n-1)b-a\big)^{2- \frac{2}{p}}  & \mbox{if} & (n-2)b \geq 2a.
\end{array}
\right.
\]
Both upper bounds are optimal for $p = 1$.
\end{corollary}

\begin{proof}
This is an easy consequence of Theorem \ref{T:3-norm-A-p} and the fact that $\|A\|_p = \|A\|_q$ where $1/p + 1/q = 1$.
\end{proof}

\section{Concluding remarks}

\begin{enumerate}
\item The estimations provided in Theorem \ref{T:3-norm-A-p} are more accurate than those presented in Theorem \ref{T:2-norm-A-p}. However, neither of them is the precise value of $\|A\|_p$. Is it possible to provide a simple closed formula for $\|A\|_p$?

\item We also encountered the following description of $A(n,-a,b)$ which comes from harmonic analysis. This observation provides another upper bound for $\|A\|_p$. Let $\mathcal{P}_{n}$ denote the space of polynomials of degree at most $n$. We can write
\[
A(n,-a,b) = -(a+b)I + bK
\]
where $K$ is the $n\times n$ all-ones matrix. We may interpret $K$ as an operator on $\mathcal{P}_{n-1}$. More explicitly, for each polynomial $f(z)=a_0+a_1z+\cdots+a_{n-1}z^{n-1} \in \mathcal{P}_{n-1}$, we have
\[
(Kf)(z) = (a_{0}+a_{1}+\cdots+a_{n-1}) \varphi(z),
\]
where
\[
\varphi(z)=1+z+\cdots+z^{n-1}.
\]
Note that as a consequence of this interpretation, i.e., making a correspondence between $f \in \mathcal{P}_{n-1}$ and the vector $(a_0,a_1,\dots,a_{n-1}) \in \mathbb{R}^n$, we have
\[
\|f\|_p = (|a_0|^p+|a_1|^p+\dots+|a_{n-1}|^p)^{1/p}.
\]
Moreover, as another integral representation, it is also straightforward to see that
\[
(Kf)(z) = \varphi(z) \int_{0}^{2\pi} f(e^{i\theta}) \overline{\varphi(e^{i\theta})} \, \frac{d\theta}{2\pi}.
\]
Therefore, we immediately see that
\begin{eqnarray*}
\|K\|_{p} &=& \sup_{f\in \mathcal{P}_{n-1}} \frac{\|Kf\|_{p}}{\|f\|_{p}}\\
&=& \sup_{f\in \mathcal{P}_{n-1}} \frac{\|\varphi\|_{p}}{\|f\|_{p}} \left| \int_{0}^{2\pi} f(e^{i\theta}) \overline{\varphi(e^{i\theta})} \, \frac{d\theta}{2\pi} \right|.
\end{eqnarray*}
Now, on one hand, $\|\varphi\|_{p} = n^{1/p}$ and, on the other hand,
\begin{eqnarray*}
\left| \int_{0}^{2\pi} f(e^{i\theta}) \overline{\varphi(e^{i\theta})} \, \frac{d\theta}{2\pi} \right|
&\leq& \int_{0}^{2\pi} |f(e^{i\theta}) \varphi(e^{i\theta})| \, \frac{d\theta}{2\pi}\\
&\leq& \int_{0}^{2\pi} \|f\|_{p} n^{1/q}|\varphi(e^{i\theta})| \, \frac{d\theta}{2\pi}\\
&=& n^{1/q} \|f\|_{p} \|\varphi\|_{L^1(\mathbb{T})},
\end{eqnarray*}
where
\[
\|\varphi\|_{L^1(\mathbb{T})} = \int_{0}^{2\pi} |\varphi(e^{i\theta})| \, \frac{d\theta}{2\pi}.
\]
Therefore,
\[
\|K\|_{p} \leq n \|\varphi\|_{L^1(\mathbb{T})},
\]
which gives the upper estimate
\[
\|A(n,-a,b)\|_p \leq a+b + bn \|\varphi\|_{L^1(\mathbb{T})}.
\]
Does the above estimation, or its variations, lead to a precise formula for $\|A\|_p$?
\end{enumerate}

\section*{Funding}

This work was partially supported by a research grant from NSERC (Canada), a scholarship from FRQNT (Qu\'ebec), a Ramanujan Fellowship, and a SwarnaJayanti Fellowship from SERB and DST (Govt. of India).

\section*{Acknowledgements}
The authors are grateful to the anonymous referee for the valuable comments, and constructive remarks that improved the quality of the article.

\bibliographystyle{tfnlm}
\bibliography{biblio}

\end{document}